\documentclass[12pt]{amsart}
\usepackage{setspace}
\usepackage{latexsym}
\usepackage{amssymb}
\usepackage{amsthm}
\usepackage{amsfonts}
\usepackage{amsmath}
\usepackage{amscd}
\usepackage{mathrsfs}
\usepackage{enumerate}
\usepackage{url, hyperref, doi}
\hypersetup{citecolor=blue, linkcolor=blue, colorlinks=true}
\usepackage{longtable, float, booktabs, cancel, cite, tikz, pgfmath, mathtools}
\usepackage[margin=1in]{geometry}
\usepackage{colonequals}
\usepackage{verbatim}

\newcommand{\rthmname}{Theorem}
\newtheorem*{rthm}{\rthmname}
\newenvironment{restatethm}[1]{\renewcommand{\rthmname}{Theorem #1}
\begin{rthm}
}{\end{rthm}}
\newtheorem{thm}{Theorem}[section]
\newtheorem{cor}[thm]{Corollary}
\newtheorem{conj}[thm]{Conjecture}
\newtheorem{lem}[thm]{Lemma}

\newtheorem{prop}[thm]{Proposition}
\theoremstyle{definition}\newtheorem{defn}[thm]{Definition}
\theoremstyle{definition}\newtheorem{rem}[thm]{Remark}
\theoremstyle{definition}\newtheorem{ex}[thm]{Example}
\numberwithin{equation}{thm}

\renewcommand\le{\leqslant}
\renewcommand\ge{\geqslant}

\newcommand{\N}{\mathbb{N}}

\newcommand{\Z}{\mathbb{Z}}

\newcommand{\id}{{\rm id}}
\newcommand{\LCM}{{\rm LCM}}
\newcommand{\fix}{{\rm fix}}
\newcommand{\sgn}{\rm sgn}
\begin{document}

\title{Circular sorting in the alternating group}

\author{Melanie J. Ferreri}
\address[M. J. Ferreri]{Department of Mathematics, William \& Mary, Williamsburg, VA 23187}
\email{\textcolor{blue}{\href{mailto:mjferreri@wm.edu}{mjferreri@wm.edu}}}

\author{Eric Swartz}
\address[E. Swartz]{Department of Mathematics, William \& Mary, Williamsburg, VA 23187}
\email{\textcolor{blue}{\href{mailto:easwartz@wm.edu}{easwartz@wm.edu}}}

\author{Nicholas J. Werner}
\address[N. J. Werner]{Department of Mathematics, Computer and Information Science, SUNY at Old Westbury, Old Westbury, NY 11568}
\email{\textcolor{blue}{\href{mailto:wernern@oldwestbury.edu}{wernern@oldwestbury.edu}}}

\begin{abstract}
    The symmetric group $S_n$ is generated by transpositions, and problems of sorting permutations using transpositions are well studied. In recent work, Adin, Alon, and Roichman 
    studied the related problem of sorting $n$ points on a circle, and gave a formula for the maximum number of adjacent swaps required. This is equivalent to the number of adjacent transpositions required to transform any permutation into a power of the cyclic permutation $(1,2,\ldots, n)$.
    
    The focus of this work is an analogous question in the alternating group $A_n$, which is generated by $3$-cycles. That is, using 3-cycles instead of transpositions, what is the maximum number of steps required to transform an even permutation into a power of $(1,2,\ldots, n)$ in the alternating group? We determine this number exactly for even $n$ and $n \equiv 1 \pmod{4}$. For $n \equiv 3 \pmod{4}$, we show that the sorting number can take one of two possible values and give explicit constructions demonstrating that the larger value occurs infinitely often.
\end{abstract}

\maketitle

\section{Introduction}\label{sec: introduction}

\noindent\textbf{Notation and Conventions}. Throughout, $n$ is a positive integer greater than 1 and $\Z_n := \Z/n\Z = \{0, 1, \ldots, n-1\}$ is the ring of integers modulo $n$. The unit group of $\Z_n$ is denoted by $\Z_n^\times$. As usual, $S_n$ and $A_n$ are the symmetric and alternating groups, respectively, of degree $n$. We will assume that permutations from these groups act on $\Z_n$ rather than on $\{1, \ldots, n\}$. Permutations are composed from left to right, so for example $(0 , 1)(1 ,2) = (0 , 2 , 1)$. We use $\id$ to denote the identity permutation.\\

As is well known, the symmetric group $S_n$ is generated by 2-cycles, or transpositions. A classical problem of permutations is as follows: Given a permutation $\pi$ of $n$ elements, how many simple transpositions would need to be applied to $\pi$ in order to produce the identity permutation? That is, if we write down the one-line notation of $\pi$ and apply a sequence of transpositions, how many transpositions could we possibly need to get all the numbers back into ``natural'' order?  

Sorting problems of this kind are common in the literature \cite{GyoriTuran, FengETAL, vanZuylenETAL, ChenSkiena} and have applications to molecular biology and studies of gene sequences \cite{BafnaPevzner, Christie, EliasHartman}. One variation that has received recent attention is the \emph{circular} sorting problem, where instead of sorting numbers on a line, we sort numbers on a circle. In this scenario, the numbers are considered sorted as long as they appear in their natural order clockwise on a circle, up to rotations. Let $c = (0,1,\ldots, n-1)$. Then, the elements of $\langle c \rangle$ form the equivalence class (under rotation) of the trivial cyclic permutation. 
Given a permutation $\pi$ in $S_n$, we define the \emph{sorting number of $\pi$} to be the minimum number of two-cycles needed to apply to $\pi$ in order to transform it into some power of $c$.
The circular sorting problem is then to find the maximal sorting number over all
permutations in $S_n$. 
This was studied by Adin, Alon, and Roichman in \cite{AdinAlonRoichman2025}, where they gave exact sorting numbers when the only transpositions allowed are those involving adjacent elements. The variation on this problem in which non-adjacent swaps are allowed was also considered in this work, as well as  by Bastide, Bishnoi, Groenland, Gijswijt, and Joshi in \cite{bastide2025circularsortingstrongcomplete}, who provided new lower bounds. 

In the present work, we examine an analogous sorting problem in which the permutation $\pi$ lies in the alternating group $A_n$, and we wish to resolve a circular sort by applying 3-cycles instead of transpositions. 

\begin{defn}\label{def: h}
When $n$ is odd, we let $c=(0 , 1 , \ldots , n-1)$, and when $n$ is even we take $c=(0 , 1 , \ldots , n-1)^2$. This ensures that $c \in A_n$ in all cases. Given $\pi \in A_n$, we denote the coset $\langle c \rangle \pi$ by $[\pi]$.  We define $h(\pi)$ to be the minimum number of 3-cycles in a decomposition of $\pi$ into $3$-cycles, with the convention that $h(\id) = 0$. Also, we define
\begin{equation*}
h([\pi]) = \min_{\sigma \in [\pi]} h(\sigma)  = \min_{0 \le j \le n-1} h(c^j \pi) \quad \text{ and } \quad h(n) = \max_{\pi \in A_n} h([\pi]).
\end{equation*}
\end{defn}

With this notation, the circular sorting number of $\pi$ via 3-cycles is $h([\pi])$. This can be computed by examining the elements of the coset $[\pi]=\langle c \rangle\pi$, and our goal is to determine $h(n)$, which is the maximum possible circular sorting number among all elements of $A_n$.

\begin{ex}\label{ex: sorting}
We give an example of circular sorting with 3-cycles. Begin with the permutation of $\Z_7$ corresponding to $\pi = (0 , 1 , 4 , 5 , 6 , 3 , 2)$.
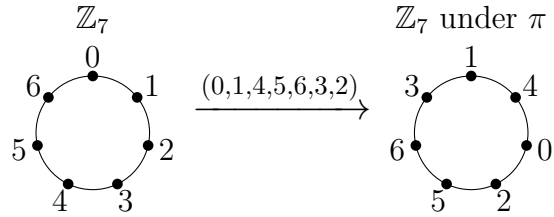
\begin{center}
\def\perm{{1,4,0,2,5,6,3}}
\begin{tikzpicture}
\node at (0,1.5) {$\Z_7$};
\draw (0,0) circle (0.75);
\foreach \x in {0,1,...,6} \node[circle,fill=black!100, inner sep = 1pt, minimum size = 4pt] 
at ({-0.75*cos(360*\x/7+90)},{0.75*sin(360*\x/7+90)}) {};

\foreach \x in {0,1,...,6} \node at ({-cos(360*\x/7+90)},{sin(360*\x/7+90)}) {$\x$};

\node at (2.5,0.5) {\large$\xrightarrow{(0 , 1 , 4 , 5 , 6 , 3 , 2)}$};

\node at (5,1.5) {$\Z_7$ under $\pi$};
\draw (5,0) circle (0.75);
\foreach \x in {0,1,...,6} \node[circle,fill=black!100, inner sep = 1pt, minimum size = 4pt] 
at ({-0.75*cos(360*\x/7+90)+5},{0.75*sin(360*\x/7+90)}) {};

\foreach \x in {0,1,...,6} \node at ({-cos(360*\x/7+90)+5},{sin(360*\x/7+90)}) {\pgfmathparse{\perm[\x]}\pgfmathresult};
\end{tikzpicture}
\end{center}
Here, $c=(0 , 1 , \ldots , 6)$. Compute the elements of the coset $[\pi] = \langle c \rangle \pi$ and decompose each element into 3-cycles. This gives
\begin{align*}
[\pi] &= \{(0 , 1 , 4 , 5 , 6 , 3 , 2), (0 , 4 , 6 , 1)(3 , 5), (1 , 2 , 5)(3 , 6 , 4), (0 , 2 , 6)(1 , 5 , 4),\\
&\hspace*{0.5in} (0 , 5)(1 , 6 , 2 , 3), (0 , 6 , 5 , 2 , 1 , 3 , 4), (0 , 3)(2 , 4) \}\\
&=\{(0 , 1 , 4)(0 , 5 , 6)(0 , 3 , 2),  (0 , 4 , 6)(0 , 1 , 3)(3 , 0 , 5), (1 , 2 , 5)(3 , 6 , 4), (0 , 2 , 6)(1 , 5 , 4),\\
&\hspace*{0.5in} (0 , 5 , 1)(1 , 0 , 6)(1 , 2 , 3), (0 , 6 , 5)(0 , 2 , 1)(0 , 3 , 4), (0 , 3 , 2)(2 , 0 , 4)\}
\end{align*}
Each of these decompositions into 3-cycles uses the minimum possible number of 3-cycles; see Lemma \ref{lem: computing h} and Remark \ref{rem: computing h}. We have $h([\pi])=h(c^2\pi)=2$, so we can use $(c^2\pi)^{-1}=(3 , 4 , 6)(1 , 5 , 2)$ to circularly sort $\Z_7$ and resolve the permutation done by $\pi$.
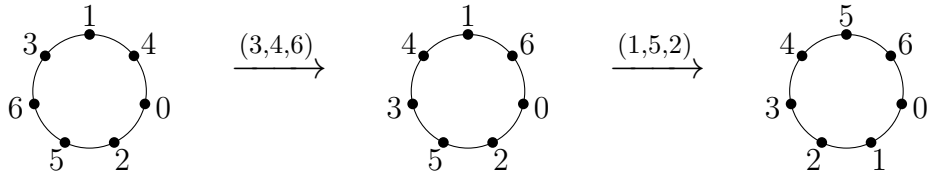
\begin{center}
\def\perm{{1,4,0,2,5,6,3}}
\def\tcone{{1,6,0,2,5,3,4}}
\def\tctwo{{5,6,0,1,2,3,4}}
\begin{tikzpicture}
\draw (0,0) circle (0.75);
\foreach \x in {0,1,...,6} \node[circle,fill=black!100, inner sep = 1pt, minimum size = 4pt] 
at ({-0.75*cos(360*\x/7+90)},{0.75*sin(360*\x/7+90)}) {};

\foreach \x in {0,1,...,6} \node at ({-cos(360*\x/7+90)},{sin(360*\x/7+90)}) {\pgfmathparse{\perm[\x]}\pgfmathresult};

\node at (2.5,0.5) {\large$\xrightarrow{(3 , 4 , 6)}$};

\draw (5,0) circle (0.75);
\foreach \x in {0,1,...,6} \node[circle,fill=black!100, inner sep = 1pt, minimum size = 4pt] 
at ({-0.75*cos(360*\x/7+90)+5},{0.75*sin(360*\x/7+90)}) {};

\foreach \x in {0,1,...,6} \node at ({-cos(360*\x/7+90)+5},{sin(360*\x/7+90)}) {\pgfmathparse{\tcone[\x]}\pgfmathresult};

\node at (7.5,0.5) {\large$\xrightarrow{(1 , 5 , 2)}$};

\draw (10,0) circle (0.75);
\foreach \x in {0,1,...,6} \node[circle,fill=black!100, inner sep = 1pt, minimum size = 4pt] 
at ({-0.75*cos(360*\x/7+90)+10},{0.75*sin(360*\x/7+90)}) {};

\foreach \x in {0,1,...,6} \node at ({-cos(360*\x/7+90)+10},{sin(360*\x/7+90)}) {\pgfmathparse{\tctwo[\x]}\pgfmathresult};

\end{tikzpicture}
\end{center}
In this example, the sorting step was done using two 3-cycles. Since both $c^3\pi$ and $c^6\pi$ can also be written using two 3-cycles, either of these permutations could have been used in place of $c^2\pi$ to perform the sorting step.
\end{ex}

We provide solutions for the sorting problem in $A_n$ for all $n$ which are congruent to $0,1,$ or $2 \pmod 4$. When $n \equiv 3 \pmod{4}$, we prove that there are only two possible values for $h(n)$. 

\begin{restatethm}{\ref{thm: n is even}}
Assume that $n$ is even and $n \ge 4$. Then, $h(n) = n/2$.
\end{restatethm}

\begin{restatethm}{\ref{thm: n=1 mod 4}}
Assume that $n \equiv 1 \pmod{4}$. Then, $h(n)=(n-1)/2$.
\end{restatethm}

\begin{restatethm}{\ref{thm: n=3 mod 4 bounds}}
Assume that $n \equiv 3 \pmod{4}$. Then, $h(n) = (n-3)/2$ or $h(n)=(n-1)/2$.
\end{restatethm}

    These theorems indicate that circular sorting by 3-cycles in $A_n$ is better behaved than circular sorting by transpositions in $S_n$. While the circular sorting number $t(n)$ for $S_n$ has been calculated for many values of $n$ \cite{bastide2025circularsortingstrongcomplete}, currently the best known 
    bounds for general $n$ are
\begin{equation*}
n-\lceil e \cdot( \ln n+1 ) \rceil \le t(n) \le n-2;
\end{equation*}
see \cite[Observation 3.1,Theorem 3.10]{AdinAlonRoichman2025}. In contrast, the circular sorting number for $A_n$ is always one of $(n-3)/2$, $(n-1)/2$, or $n/2$. (It should be noted that if sorting is done using only \emph{adjacent} transpositions, then the sorting number for $S_n$ is exactly $\left \lfloor (n-1)^2/4 \right \rfloor$ \cite[Theorem 1.1]{AdinAlonRoichman2025}.)

Theorems \ref{thm: n is even}, \ref{thm: n=1 mod 4}, and \ref{thm: n=3 mod 4 bounds} are proved in Section \ref{sec: main theorems}. The situation in which $n \equiv 3 \pmod 4$ is addressed in more detail in Sections \ref{sec: general n=3 mod 4} and \ref{sec: specific n=3 mod 4}. Although we do not know $h(n)$ exactly in all cases for which $n \equiv 3 \pmod{4}$, we are able to determine the sorting number for values of $n$ with certain prime factorizations.

\begin{restatethm}{\ref{thm:pqfulldiam}}
 Let $n = pq$, where $p$ and $q$ are distinct odd primes. Then, $h(n) = (n-1)/2$.
\end{restatethm}

\begin{restatethm}{\ref{thm: exponent sum}}
Let $n = p_1^{c_1} \cdots p_k^{c_k}$, where $p_1, \dots, p_k$ are distinct odd primes and each $c_i \ge 1$. Assume that either of the following two conditions holds:
\begin{enumerate}[(i)]
\item $c_1 + \cdots + c_k$ is even.
\item $c_1 + \cdots + c_k$ is odd and there exists $1 \le \ell \le k$ such that $p_\ell \equiv 1 \pmod 4$ and $c_\ell$ is odd.
\end{enumerate}
Then, $h(n) = (n-1)/2$.
\end{restatethm}

Our general approach to calculating $h(n)$ is to construct a \emph{witness} permutation for which all rotations require the maximal number of 3-cycles to sort. For values of $n$ not covered by Theorems \ref{thm: n is even}, \ref{thm: n=1 mod 4}, or \ref{thm: exponent sum}, we determine some sorting numbers computationally in Section \ref{sec: specific n=3 mod 4}. Here, some of the witnesses are found via exhaustive search with GAP \cite{GAP4}; others are motivated by constructions for orthomorphisms with specific cycle type given by Bors and Wang in \cite{BorsWang2022}. The smallest value of $n$ left unresolved by these methods was $n=31$. In Section \ref{sec: Claude}, using Claude Fable 5 \cite{claudefable2026}, Claude Opus 5 \cite{claudeopus2026}, and GAP for verification, new witnesses were found for this and other values of $n$. 
Moreover, Claude suggested a theorem (Theorem \ref{thm:compositewitnesstype}) that allows us to prove that $h(n)=(n-1)/2$ for additional values of $n$. All work in this paper completed with the assistance of generative AI is discussed in Section \ref{sec: Claude}.

\section{Main theorems}\label{sec: main theorems}

Recall that for $\pi \in A_n$, $h(\pi)$ is the minimum number of 3-cycles needed to express $\pi$ as a product of 3-cycles. We begin with some methods to calculate $h(\pi)$.

\begin{lem}\label{lem: computing h}\mbox{}
\begin{enumerate}[(1)]
\item Let $\pi \in A_n$ and let $m \ge 0$ be the number of cycles of odd length in a decomposition of $\pi$ into disjoint cycles. Consider each fixed point of $\pi$ to be a cycle of length 1, so that each fixed point contributes 1 towards $m$. Then, $h(\pi)=(n - m)/2$.
\item If $n$ is odd, then $h(n) \le (n-1)/2$.
\item If $n$ is even, then $h(n) \le n/2$.
\end{enumerate}
\end{lem}
\begin{proof}
Given any cycle $\sigma$, let $L(\sigma)$ denote the length of $\sigma$. Express $\pi$ as a product of $t$ disjoint cycles $\sigma_1, \ldots, \sigma_t$, where we include a cycle of length 1 for each fixed point of $\pi$. With this convention, $\sum_{i=1}^t L(\sigma_i) = n$. By \cite[Corollary 2.4]{HerzogReid}, $h(\pi) = \sum_{i=1}^t \big\lceil(1/2)(L(\sigma_i)-1)\big\rceil$. Then,
\begin{align*}
h(\pi) &= \sum_{\substack{\sigma \text{ in } \pi,\\L(\sigma) \text{ odd}}} \dfrac{L(\sigma)-1}{2} + \sum_{\substack{\sigma \text{ in } \pi,\\L(\sigma) \text{ even}}} \dfrac{L(\sigma)}{2}\\
&=\dfrac{1}{2}\Big(\sum_{\sigma \text{ in } \pi} L(\sigma) - m\Big)\\
&=\dfrac{1}{2}(n-m),
\end{align*}
which proves (1).

For (2) note that when $n$ is odd, each permutation in $A_n$ must contain at least one cycle of odd length in any decomposition into disjoint cycles. Hence, $h(n) = \max_{\pi \in A_n} h([\pi]) \le (n-1)/2$. Similarly, when $n$ is even, it is possible for a permutation in $A_n$ to contain no cycles of odd length, so $h(n) \le n/2$ and (3) holds.
\end{proof}

\begin{rem}\label{rem: computing h}
To express an even permutation as a product of 3-cycles in an optimal way, we can use the following well-known decompositions.
\begin{enumerate}
\item Let $\sigma = (a_1 , a_2 , \ldots ,a_{2k+1})$ be a cycle of odd length $2k+1$, where $k \ge 1$. Then,
\begin{equation*}
\sigma = (a_1 , a_2 , \ldots a_{2k+1}) = (a_1 , a_2 , a_3)(a_1 , a_4 , a_5) \cdots (a_1 , a_{2k} , a_{2k+1}),
\end{equation*}
and $h(\sigma) = k$.

\item Let $\sigma = (a_1 , a_2 , \ldots , a_{\ell_1})(b_1 , b_2 , \ldots , b_{\ell_2})$ be a product of two disjoint cycles of even lengths $\ell_1 \ge 2$ and $\ell_2 \ge 2$. Form the following three permutations:
\begin{align*}
\sigma_1 &= (a_1 , a_2 , a_3)(a_1 , a_4 , a_5) \cdots (a_1 , a_{\ell_1-2} ,  a_{\ell_1-1}),\\
\sigma_2 &= (a_1 , a_{\ell_1} , b_1)(b_1 , a_1 , b_2), \text{ and}\\
\sigma_3 &= (b_1 , b_3 , b_4) (b_1 , b_5 , b_6) \cdots (b_1 , b_{\ell_2-1} ,  b_{\ell_2}).
\end{align*}
Then, $\sigma = \sigma_1 \sigma_2 \sigma_3$ and $h(\sigma) = (\ell_1+\ell_2)/2$.
\end{enumerate}
For a general permutation $\pi \in A_n$, we can first write $\pi$ as a product of disjoint cycles, and then apply the above decompositions. By \cite[Corollary 2.4]{HerzogReid}, this produces an expression for $\pi$ that uses the minimum possible number of 3-cycles. Further results on minimal decompositions of permutations into $k$-cycles can be found in \cite{HerzogReid} and \cite{HerzogReid2}.
\end{rem}

We will show that the upper bounds for $h(n)$ in Lemma \ref{lem: computing h} are always attained if $n$ is even or $n \equiv 1 \pmod{4}$. The case where $n \equiv 3 \pmod{4}$ is more delicate, and here our most general result is that $(n-3)/2 \le h(n) \le (n-1)/2$. Nevertheless, in Section \ref{sec: general n=3 mod 4} we will demonstrate that there are infinitely many $n \equiv 3 \pmod{4}$ for which $h(n)=(n-1)/2$.

\begin{prop}\label{prop: n is even}
Assume that $n$ is even. Let $\pi \in A_n$ be such that for all $x \in \Z_n$, $x^\pi \not\equiv x \pmod{2}$ (that is, for each $x \in \Z_n$, $x$ and $x^\pi$ have different parities). Then, $h([\pi]) = n/2$.
\end{prop}
\begin{proof}
Here, $c=(0 \, 1 \, \ldots \, n-1)^2$ and each element of $[\pi]$ has the form $c^j\pi$ for some $j \ge 0$. For each $x \in \Z_n$, we have $x^{c^j\pi} = (x+2j)^\pi$. Since $x$ and $x+2j$ have the same parity and $\pi$ switches parities, this means that $x$ and $x^{c^j\pi}$ have different parities. So, if $x^{(c^j\pi)^k} = x$, then $k$ must be even. Thus, each $x$ is in a cycle of $c^j \pi$ of even length. In other words, $c^j \pi$ contains no cycles of odd length. By Lemma \ref{lem: computing h}(1) with $m = 0$, $h(c^j \pi) = n/2$ for all $j$.
\end{proof}

\begin{thm}\label{thm: n is even}
Assume that $n$ is even and $n \ge 4$. Then, $h(n) = n/2$.
\end{thm}
\begin{proof}
Let $\pi=(0 , 1)(2 , 3 , \ldots , n-1) \in A_n$. Then, $x^\pi \not\equiv x \pmod{2}$ for all $x \in \Z_n$. By Lemma \ref{lem: computing h} and Proposition \ref{prop: n is even}, $h(n) = h([\pi])=n/2$.
\end{proof}

\begin{prop}\label{prop: n=1 mod 4}
Assume that $n \equiv 1 \pmod{4}$. Let $\pi \in S_n$ be such that $x^\pi = -x $ for all $x \in \Z_n$. Then, $\pi$ is an even permutation and $h([\pi]) = (n-1)/2$.
\end{prop}
\begin{proof}
First, when $\pi$ is written as a product of disjoint cycles, it contains $(n-1)/2$ transpositions. Since $n \equiv 1 \pmod{4}$, $\pi$ is even. 

Next, $c=(0 , 1 , \ldots , n-1)$ because $n$ is odd. We claim that for all $j \ge 0$, $c^j \pi$ has order 2 and that $(\tfrac{n-1}{2})j \in \Z_n$ is the unique fixed point of $c^j \pi$. To see this, let $x \in \Z_n$. Then,
\begin{equation*}
x^{c^j \pi} = (x+j)^\pi = -x-j.
\end{equation*}
Since $n>2$, $c^j \pi$ cannot be the identity function. Furthermore,
\begin{equation*}
x^{(c^j \pi)^2} = (-x-j)^{c^j\pi} = (-x)^\pi = x,
\end{equation*}
so $|c^j \pi| = 2$.  

For the fixed point, first note that $(\tfrac{n-1}{2})^\pi = \frac{n+1}{2}$. So, 
\begin{align*}
\big((\tfrac{n-1}{2})j\big)^{c^j \pi} = \big((\tfrac{n-1}{2})j+j\big)^{\pi} = \big((\tfrac{n+1}{2})j\big)^\pi = (\tfrac{n-1}{2})j.
\end{align*}
This fixed point is unique for $c^j \pi$, because if $y^{c^j \pi} = y$, then $-y-j=y$. Solving this equation for $y \in \Z_n$ yields $y = (-2)^{-1}j = -(\tfrac{n+1}{2})j = (\tfrac{n-1}{2})j$.

Finally, since $c^j \pi$ has order 2, when the permutation is written as a product of disjoint cycles it contains no cycle of odd length greater than 1. Since $c^j \pi$ has exactly one fixed point, $h(c^j \pi)=(n-1)/2$ by Lemma \ref{lem: computing h}. This holds for all $j \ge 0$, so $h([\pi])=(n-1)/2$.
\end{proof}

\begin{thm}\label{thm: n=1 mod 4}
Assume that $n \equiv 1 \pmod{4}$. Then, $h(n)=(n-1)/2$.
\end{thm}
\begin{proof}
Apply Proposition \ref{prop: n=1 mod 4} and Lemma \ref{lem: computing h}.
\end{proof}

\begin{ex} We demonstrate the construction in Proposition \ref{prop: n=1 mod 4} when $n=5$.
    Let $\pi = (0)(1,4)(2,3)$. Then 
         \begin{align*}
       \langle c \rangle \pi= \{&\pi = (0)(1,4)(2,3), 
        c\pi = (0,4)(1,3)(2), 
        c^2\pi = (0,3)(1,2)(4),\\
        &c^3\pi = (0,2)(1)(3,4),
        c^4\pi = (0,1)(2,4)(3)
        \}.
    \end{align*}
    Every permutation in $\langle c \rangle \pi$ contains exactly one cycle of odd length, so $h([\pi]) = \frac{5-1}{2} =2$ by Lemma \ref{lem: computing h}. Since $h(5) \le 2$, we have $h(5) =2$.
\end{ex}

It remains to consider the case where $n \equiv 3 \pmod{4}$. For such $n$, we are able to prove that $h(n)$ is equal to either $(n-3)/2$ or $(n-1)/2$. Towards this end, we will demonstrate that $A_n$ always contains a permutation $\pi$ such that $h([\pi]) = (n-3)/2$.

\begin{prop}\label{prop: n=3 mod 4}
Assume that $n \equiv 3 \pmod{4}$ and $n \ge 7$. Define a permutation $\pi$ on $\Z_n$ by
\begin{equation*}
x^\pi = \begin{cases} x, & x=\tfrac{n-1}{2} \text{ or } \tfrac{n+1}{2}\\ -x, & x \ne \tfrac{n-1}{2} \text{ or } \tfrac{n+1}{2}
\end{cases}.
\end{equation*}
Then,
\begin{enumerate}[(1)]
\item $\pi \in A_n$ and $h(\pi) = (n-3)/2$;
\item for $1 \le j \le n-1$, $h(c^j \pi) = (n-1)/2$;
\item $h([\pi]) = (n-3)/2$.
\end{enumerate}
\end{prop}
\begin{proof}
(1) Clearly, $\pi^2 = \id$, so $|\pi|=2$ because $n > 2$. The fixed points of $\pi$ in $\Z_n$ are $0$, $(n-1)/2$, and $(n+1)/2$. So, when written as a product of disjoint cycles, $\pi$ consists of $(n-3)/2$ transpositions and has three fixed points. By Lemma \ref{lem: computing h}, $h(\pi) = (n-3)/2$.

(2) We will consider three cases depending on $j$, and determine the cycle structure of $c^j \pi$ in each case.\\

\noindent\textbf{Case 1}: $j=1$.\\
As a function on $\Z_n$, we have
\begin{align*}
x^{c\pi} = (x+1)^\pi &= \begin{cases} x+1, & \text{if } x+1=\tfrac{n-1}{2} \text{ or } \tfrac{n+1}{2}\\ -x-1, & \text{if } x+1 \ne  \tfrac{n-1}{2} \text{ or } \tfrac{n+1}{2}
\end{cases}\\
&= \begin{cases}
x+1, & \text{if } x = \tfrac{n-3}{2} \text{ or } \tfrac{n-1}{2}\\
-x-1, & \text{if } x \ne \tfrac{n-3}{2} \text{ or } \tfrac{n-1}{2}
\end{cases}.
\end{align*}
If $c\pi$ has a fixed point $y$, then $y^{c\pi} = -y-1$. This implies that $y = \tfrac{n-1}{2}$, which in turn means $y^{c\pi} = y+1$. This is a contradiction, so $c\pi$ has no fixed point. Next, one may check that $c \pi$ contains the 3-cycle $\big( \tfrac{n-3}{2}, \, \tfrac{n-1}{2}, \, \tfrac{n+1}{2}\big)$. Moreover, when $x$ is not equal to any entry in this 3-cycle, we have
\begin{align*}
x^{(c\pi)^2} = (-x-1)^{c\pi}=(-x)^{\pi} = -(-x)=x.
\end{align*}
It follows that $c\pi$ consists of a single 3-cycle and $(n-3)/2$ transpositions. Thus, $h(c\pi) = (n-1)/2$.\\

\noindent\textbf{Case 2:} $j=n-1$.\\
This is similar to Case 1. Here, for $x \in \Z_n$,
\begin{align*}
x^{c^{-1}\pi} = (x-1)^\pi &= \begin{cases} x-1, & \text{if } x-1=\tfrac{n-1}{2} \text{ or } \tfrac{n+1}{2}\\ -x+1, & \text{if }  x-1 \ne  \tfrac{n-1}{2} \text{ or } \tfrac{n+1}{2}
\end{cases}\\
&= \begin{cases}
x-1, & \text{if } x = \tfrac{n+1}{2} \text{ or } \tfrac{n+3}{2}\\
-x+1, & \text{if } x \ne \tfrac{n+1}{2} \text{ or } \tfrac{n+3}{2}
\end{cases}.
\end{align*}
Arguing as in Case 1 shows that $c^{-1}\pi$ consists of the 3-cycle $\big(\tfrac{n+3}{2}, \, \tfrac{n+1}{2}, \, \tfrac{n-1}{2}\big)$, $(n-3)/2$ transpositions, and has no fixed points. Once again, we obtain $h(c^{-1} \pi) = (n-1)/2$.\\

\noindent\textbf{Case 3}: $2 \le j \le n-2$.\\
First, $(\tfrac{n-1}{2})j$ is the unique fixed point of $c^j \pi$. We have
\begin{equation*}
\big((\tfrac{n-1}{2})j\big)^{c^j\pi} = \big((\tfrac{n-1}{2})j+j\big)^{\pi} = \big((\tfrac{n+1}{2})j\big)^{\pi}.
\end{equation*}
Since $\tfrac{n+1}{2}$ is a unit of $\Z_n$ and $j \not\equiv \pm 1 \pmod{n}$, applying $\pi$ to $(\tfrac{n+1}{2})j$ produces $-(\tfrac{n+1}{2})j = (\tfrac{n-1}{2})j$. For uniqueness, suppose that $y \in \Z_n$ and $y=y^{c^j \pi} = (y+j)^\pi$. By the definition of $\pi$, this will equal either $y+j$ or $-y-j$. The former case implies that $j \equiv 0 \pmod{n}$, which is a contradiction. Hence, $y = -y-j$, and solving for $y$ yields $y = \big(\tfrac{n-1}{2}\big)j$.

Next, we note that $c^j \pi$ contains the following 4-cycle:
\begin{equation*}
\big( \tfrac{n-1}{2} -j, \; \tfrac{n-1}{2}, \; \tfrac{n+1}{2} -j, \; \tfrac{n+1}{2} \big).
\end{equation*}

To complete the description of the cycle structure of $c^j \pi$, let $x \in \Z_n$ such that $x$ does not occur in the above 4-cycle. Then, neither $x$ nor $x+j$ is a fixed point of $\pi$, so 
\begin{align*}
x^{c^j \pi} &= (x+j)^\pi = -x-j, \text{ and}\\
(-x-j)^{c^j \pi} &= (-x)^\pi = -(-x)=x.
\end{align*}
Thus, $x^{(c^j \pi)^2} = x$, and we conclude that either $x$ is the unique fixed point of $c^j \pi$, or $x$ occurs in a transposition of $c^j \pi$. Consequently, $c^j \pi$ contains one fixed point, one 4-cycle, and $(n-5)/2$ transpositions. Hence, $h(c^j \pi) = (n-1)/2$.

(3) This follows from (1) and (2).
\end{proof}

\begin{thm}\label{thm: n=3 mod 4 bounds}
Assume that $n \equiv 3 \pmod{4}$. Then, $h(n) = (n-3)/2$ or $h(n)=(n-1)/2$.
\end{thm}
\begin{proof}
When $n=3$, we have $[\pi]=A_3$ for all $\pi \in A_3$, so $h(3) = h(\id) = 0$. So, assume that $n \ge 7$. By Proposition \ref{prop: n=3 mod 4} and Lemma \ref{lem: computing h}, $(n-3)/2 \le h(n) \le (n-1)/2$.
\end{proof}

In the remaining sections of the paper, we will examine the case where $n \equiv 3 \pmod{4}$ in more detail. There exist examples of such $n$ for which $h(n)=(n-3)/2$. This is (trivially) true for $n=3$, and it also holds for $n=7$ and $n=11$; see Section \ref{sec: specific n=3 mod 4}. Beyond these small values of $n$, we have not found any examples for which $h(n)=(n-3)/2$. This prompts the following conjecture.

\begin{conj}\label{conj: n=3 mod 4}
Let $n \ge 13$ be odd. Then, $h(n) = (n-1)/2$.
\end{conj}

As shown by Theorem \ref{thm: n=1 mod 4}, the conjecture is true whenever $n \equiv 1 \pmod{4}$. In Section \ref{sec: general n=3 mod 4}, we will prove Theorem \ref{thm: exponent sum}, which confirms the conjecture for infinitely many $n$ such that $n \equiv 3 \pmod{4}$.

\section{General results for $n \equiv 3 \pmod{4}$}\label{sec: general n=3 mod 4}

A permutation $\pi$ on $\Z_n$ is said to be an \emph{orthomorphism} if $\pi-\id$ is also a permutation, and $\pi$ is a \emph{strong complete mapping} if both $\pi-\id$ and $\pi+\id$ are permutations. Extensive research has been done orthomorphisms and their relation to strong complete mappings on $\Z_n$ (the book \cite{Evans} by Evans is the standard reference), and they have proved to be a useful tool for determining circular sorting numbers by transpositions for $S_n$ \cite{bastide2025circularsortingstrongcomplete}.

By Theorem \ref{thm: n=3 mod 4 bounds}, when $n \equiv 3 \pmod{4}$ either $h(n) = (n-3)/2$ or $h(n)=(n-1)/2$. In this section, we show that in order to determine the value of $h(n)$ when $n \equiv 3 \pmod{4}$, it suffices to consider orthomorphisms on $\Z_n$. Using these permutations, we prove that there are infinitely many $n \equiv 3 \pmod{4}$ for which $h(n)=(n-1)/2$. This is accomplished by using linear orthomorphisms on $\Z_n$, which are those of the form $x \mapsto xa$ for some $a \in \Z_n$. Later, in Section \ref{sec: specific n=3 mod 4}, we will follow the work done in \cite{BorsWang2022} and construct piecewise linear orthomorphisms that can be used to determine $h(n)$ for some additional values of $n \equiv 3 \pmod{4}$.

\begin{defn}\label{def: fixed point set}
For each $\pi \in A_n$, let $\fix(\pi)=\{x \in \Z_n : x^\pi = x\}$ be the set of fixed points of $\pi$ in $\Z_n$.
\end{defn}

Throughout this section, assume that $n \ge 3$ is odd. So, $c=(0 , 1  , \ldots  , n-1)$.  Recall that a \textit{derangement} in $S_n$ is an element without a fixed point, i.e., $\sigma \in S_n$ is a derangement if $i^\sigma \neq i$ for all $i$.

\begin{lem}\label{lem: sum of fixed point sets} \mbox{}
\begin{enumerate}[(1)]
\item For each $x \in \Z_n$ and $\pi \in A_n$, there exists a unique $0 \le j \le n-1$ such that $x \in \fix(c^j \pi)$.

\item For every $\pi \in A_n$, $\sum_{j=0}^{n-1} |\fix(c^j \pi)| = n$.
\item If $h(n) = (n-1)/2$, then there exists $\pi \in A_n$ such that every element of $[\pi]$ has a unique fixed point.

\item If $\pi \in A_n$ and $[\pi]$ contains a permutation with at least two fixed points, then $h([\pi]) \le (n-3)/2$.

\item If $\pi \in A_n$ and $[\pi]$ contains a derangement, then $h([\pi]) \le (n - 3)/2$.
\end{enumerate}
\end{lem}
\begin{proof}
(1) Given $x \in \Z_n$, let $f_x$ be the number of integers $j$ such that $0 \le j \le n-1$ and $x \in \fix(c^j \pi)$. We will show that $f_x=1$.

Let $j=x^{\pi^{-1}}-x$. Then, $x^{c^j \pi} = (x+j)^{\pi}=x$. So, $f_x \ge 1$. However, if $j_1$ and $j_2$ are such that $x$ is fixed by both $c^{j_1} \pi$ and $c^{j_2} \pi$, then $(x+j_1)^\pi = (x+j_2)^\pi$, and hence $j_1=j_2$. Thus, $f_x=1$. 

(2) Define $f_x$ as in the proof of (1). Then, $\sum_{j=0}^{n-1} |\fix(c^j \pi)| = \sum_{x=0}^{n-1} f_x = n$.  (This also follows from applying the Cauchy-Frobenius Lemma -- often inaccurately referred to as ``Burnside's Lemma'' -- to the action of $\langle c \rangle^\pi$ on $\{0^\pi, \dots, (n-1)^\pi\}$; see \cite[Theorem 1.7A]{DixonMortimer1996}.)

(3) Assume that $h(n) = (n-1)/2$. Then, there exists $\pi \in A_n$ with $h([\pi])=(n-1)/2$. Since $h(\sigma) \le (n-1)/2$ for every $\sigma \in A_n$, we must have $h(c^j \pi) = (n-1)/2$ for all $0 \le j \le n-1$. So, each element of $[\pi]$ has at most one fixed point. If $|\fix(c^{j_1}\pi)|=0$ for some $j_1$, then by (2) there must exist some $j_2$ such that $|\fix(c^{j_2}\pi)|\ge 2$, which is a contradiction. 

(4) Assume that $[\pi]$ contains an element $c^j \pi$ with at least two fixed points. By Lemma \ref{lem: computing h}, a decomposition of $c^j \pi$ into disjoint cycles will contain at least two cycles of odd length. Hence, $h([\pi]) \le h(c^j \pi) \le (n-2)/2$. Since $n$ is odd, this means $h([\pi]) \le (n-3)/2$.

(5) Assume $c^j\pi \in [\pi]$ is a derangement.  By (2) and the Pigeonhole Principle, $[\pi]$ contains an element with more than one fixed point.  The result follows by (4).
\end{proof}

\begin{samepage}
\begin{lem}
 \label{lem:orthsaregood}\mbox{}
\begin{enumerate}[(1)]
\item Let $\pi \in A_n$. The following are equivalent.
\begin{enumerate}[(i)]
\item $\pi$ is an orthomorphism on $\Z_n$.
\item $\pi^{-1}$ is an orthomorphism on $\Z_n$.
\item Each element of $[\pi]$ has a unique fixed point in $\Z_n$.
\end{enumerate}

\item $h(n) = (n-1)/2$ if and only if there exists $\pi \in A_n$ such that each $c^j \pi \in [\pi]$ satisfies the following properties:
\begin{enumerate}[(i)]
\item $c^j \pi$ has a unique fixed point in $\Z_n$.
\item If $x \in \Z_n$ is not the fixed point of $c^j \pi$, then $x$ is contained in a cycle of even length.
\item When written as a product of disjoint cycles, $c^j \pi$ contains an even number of cycles of even length.
\end{enumerate}
\end{enumerate}
\end{lem}
\end{samepage}
\begin{proof}
(1) (i) $\Leftrightarrow$ (ii) Clearly, $\pi$ is a bijection if and only if the same is true for $\pi^{-1}$. For the orthomorphism condition, define $f:=\pi^{-1}-\id$ and $g:=\pi-\id$, and let $\iota$ be the map sending $x$ to $-x$. Then, 
\begin{equation*}
x^{\pi f \iota} = (x-x^{\pi})^\iota = x^g,
\end{equation*}
so $\pi f \iota=g$. Thus, $f$ is a bijection if and only if $g$ is a bijection.

$(iii) \Rightarrow (ii)$ Assume that each element of $[\pi]$ has a unique fixed point. Define $\delta: \Z_n \to \Z_n$ by $\delta(x) = x^{\pi^{-1}}-x$ (we write $\delta(x)$ instead of $x^\delta$ for the sake of readability in the equation below), which is well-defined because $\pi$ is a bijection. For each $x \in \Z_n$, $x$ is fixed by $c^{\delta(x)} \pi$, because
\begin{equation*}
x^{c^{\delta(x)}\pi} = (x+\delta(x))^{\pi} = (x^{\pi^{-1}})^\pi=x.
\end{equation*}
So, if $\delta(x_1)=\delta(x_2)$, then $x_1$ is fixed by both $c^{\delta(x_1)}\pi$ and $c^{\delta(x_2)}\pi$. Since $x_2$ is also fixed by $c^{\delta(x_2)}\pi$ and these fixed points are unique by assumption, we have $x_1=x_2$. Thus, $\delta=\pi^{-1} - \id$ is injective, and $\pi^{-1}$ is an orthomorphism.

$(ii) \Rightarrow (iii)$ Assume that $\pi^{-1}$ is an orthomorphism, let $\delta = \pi^{-1} - \id$, and let $0 \le j \le n-1$. Since $\delta$ is a bijection, there is a unique $x \in \Z_n$ such that $x^\delta=j$. For this $x$, we have $x+j=x^{\pi^{-1}}$, so $x^{c^j \pi} = (x+j)^{\pi} = x$. This shows that each element of $[\pi]$ has a fixed point. By Lemma \ref{lem: sum of fixed point sets}(2), $|\fix(c^j \pi)|=1$ for every $0 \le j \le n-1$, so each fixed point is unique for $c^j \pi$.

(2) Assume $h(n)=(n-1)/2$. Then, there exists $\pi \in A_n$ such that $h(c^j \pi)=(n-1)/2$ for all $0 \le j \le n-1$. So, by Lemma \ref{lem: computing h}, the cycle structure of each $c^j \pi$ consists of a single cycle of odd length, and all other cycles have even length. Moreover, by Lemma \ref{lem: sum of fixed point sets}(3), each $c^j \pi$ has a unique fixed point. Thus, each element of $\Z_n$ not fixed by $c^j \pi$ lies in a cycle of even length, and there must be an even number of such cycles because $c^j \pi \in A_n$.

Conversely, if such a permutation $\pi$ exists, then conditions (i)--(iii) imply that $h(c^j \pi) = (n-1)/2$ for all $0 \le j \le n-1$. Hence, $h([\pi])=(n-1)/2$, and $h(n)=(n-1)/2$ by Lemma \ref{lem: computing h}.
\end{proof}

\begin{rem}\label{rem: h(pi) and h(pi^-1)}
As mentioned in the proof of Lemma \ref{lem:orthsaregood}, for any $\pi \in A_n$ there exists an element of $[\pi]$ with a given cycle structure if and only if $[\pi^{-1}]$ contains such an element.  Since the calculation of $h$ depends only on cycle structure, we have $h([\pi]) = h([\pi^{-1}])$ for all $\pi \in A_n$.
\end{rem}

From Lemma \ref{lem:orthsaregood}, we see that in order for $h(n)$ to equal $(n-1)/2$, it is necessary that there exist $\pi \in A_n$ such that each element of $[\pi]$ has a unique fixed point, and this condition is equivalent to $\pi^{-1}$ being an orthomorphism on $\Z_n$. Thus, it will be advantageous to identify orthomorphisms in $A_n$. For the remainder of this section, we will focus on linear functions on $\Z_n$ that are orthomorphisms; linear functions are particularly advantageous since it is possible to determine the interaction of such a permutation with $c$, which can itself be defined algebraically as the map $x \mapsto x + 1$ on $\Z_n$.

\begin{lem}
 \label{lem:goodlinear}
 Let $\pi: \Z_n \to \Z_n$ be a linear function of the form $x^\pi = xa$ for some $a \in \Z_n$.  
\begin{enumerate}[(1)]
\item $\pi^{-1}$ is an orthomorphism if and only if $\gcd(a, n) = 1$ and $\gcd(a - 1, n) = 1$.

\item Assume $\gcd(a,n) = \gcd(a-1, n) = 1$.  For each $0 \le j \le n-1$, $c^j \pi$ has the same cycle structure as $\pi$.
\end{enumerate}
\end{lem}

\begin{proof}
(1) First, note that in order for $\pi$ (and hence $\pi^{-1}$) to be a bijection, it is both necessary and sufficient that $\gcd(a, n) = 1$. Indeed, if $\gcd(a,n) > 1$, then $\pi$ maps both $0$ and $n/\gcd(a,n)$ to $0$; and if $\gcd(a,n) = 1$, then $a \in \Z_n^\times$, and hence has inverse $x \mapsto xa^{-1}$.

Next, when $\pi$ is a bijection, $\delta:=\pi^{-1} - \id$ is well-defined. For $x \in \Z_n$, $x^\delta = xa^{-1} - x = x(a^{-1} - 1)$, and this is a bijection if and only if $a^{-1} - 1$ is a unit, which is true if and only if $-a(a^{-1} - 1) = a - 1$ is a unit, and this is equivalent to the condition $\gcd(a-1, n) = 1$.

(2) Note that for all $x \in \Z_n$,
\begin{equation*}
x^{c^{-1}\pi} = (x - 1)^\pi = a(x - 1) = ax - a = x^{\pi c^{-a}},
\end{equation*}
so $c^{-1}\pi = \pi c^{-a}$ and hence $c^{-k}\pi = \pi c^{-ka}$ for all $k \in \Z$.  Now, since $\gcd(a-1, n) = 1$, there exists $b \in \Z_n$ such that $c^{(a-1)b} = c$. Thus, for $0 \le j \le n-1$,
\begin{equation*}
c^j \pi = c^{(a-1)bj}\pi = c^{abj}c^{-bj}\pi = c^{abj} \pi c^{-abj}.
\end{equation*}
 Since $c^j \pi$ and $\pi$ are conjugate permutations, they have the same cycle structure.
\end{proof}

\begin{prop}
 \label{prop:goodpq}
 Let $n=pq$, where $p$ and $q$ are distinct odd primes. Let $a$ be an integer such that $2 \le a < pq$, $a$ has multiplicative order $p-1$ modulo $p$, and $a$ has multiplicative order $q-1$ modulo $q$, i.e., $a$ generates both $\Z_p^\times$ and $\Z_q^\times$. Define $\pi: \Z_n \to \Z_n$ by $x^\pi = xa$.
\begin{enumerate}[(1)]
\item $\pi$ is a permutation with one fixed point; one cycle of length $p-1$; one cycle of length $q-1$; and $\gcd(p-1, q-1)$ cycles of length $\LCM(p-1, q-1)$.

\item $\pi^{-1}$ is an orthomorphism.
\end{enumerate}
\end{prop}

\begin{proof}
(1) First, $0^\pi = 0$, so $0$ is a fixed point of $\pi$.  Next, consider the orbit of $q \in \Z_n$ under the action of $\pi$, which will be
 \[ q^{\langle \pi \rangle} = \{qa^m : m \in \Z_{>0}\}.\]
 Since $a$ is a generator of the multiplicative group modulo $p$, we have
 \[ \{a^m \pmod{p}: m \in \Z_{>0}\} = \{1, \dots, p-1\},\]
 which implies
 \[ q^{\langle \pi \rangle} = \{q, 2q, \dots, (p-1)q\}.\]
So, the multiples of $q$ form a $(p-1)$-cycle under the action of $\pi$.  By analogous reasoning, the multiples of $p$ form a $(q-1)$-cycle under the action of $\pi$.  

Finally, suppose $\gcd(x,n) = 1$ and $x = x^{\pi^m} = xa^m$ for some $m > 0$.  Then, $x(a^m - 1) = 0$, which implies $a^m = 1$ in $\Z_n$.  Consequently, $a^m \equiv 1 \pmod p$ and $a^m \equiv 1 \pmod q$, so $m$ must be a multiple of both $p-1$ and $q-1$.  This means that $m$ is a multiple of
 \[ \LCM(p-1, q-1) = \frac{(p-1)(q-1)}{\gcd(p-1,q-1)}. \] 
Since $\Z_n^\times \cong C_{p-1} \times C_{q-1}$, the exponent of $\Z_n^\times$ is $\LCM(p-1,q-1)$, and so $a$ must have multiplicative order exactly $\LCM(p-1,q-1)$.  Thus, the remaining
 \[ pq - (p-1) - (q-1) - 1 = (p-1)(q-1)\]
 elements of $\Z_n$ form $\gcd(p-1,q-1)$ distinct $\LCM(p-1,q-1)$-cycles under the action of $\pi$, and $\pi$ has the stated cycle structure.

(2) By assumption, $\gcd(a,n)=1$. If $\gcd(a - 1, n) > 1$, then either $p$ or $q$ divides $a - 1$. But then either $a \equiv 1 \pmod p$ or $a \equiv 1 \pmod q$, contradicting the fact that $a$ is a generator of the multiplicative groups modulo $p$ and modulo $q$. So, $\gcd(a - 1, n) =1$ and $\pi^{-1}$ is an orthomorphism by Lemma \ref{lem:goodlinear}(1).
\end{proof}

\begin{thm}
 \label{thm:pqfulldiam}
 Let $n = pq$, where $p$ and $q$ are distinct odd primes. Then, $h(n) = (n-1)/2$.
\end{thm}

\begin{proof}
By the Chinese Remainder Theorem there exists $a \in \Z_n$ such that $a$ is a generator of the multiplicative groups of both $\Z_p$ and $\Z_q$. By Proposition \ref{prop:goodpq}, $\pi: \Z_n \to \Z_n$ defined by $x^\pi = xa$ is a permutation of $\Z_n$ such that $\pi^{-1}$ is an orthomorphism and $\pi$ consists of cycles of length $1$, $p-1$, $q-1$, and $\gcd(p-1,q-1)$ cycles of length $\LCM(p-1,q-1)$.  Since both $\gcd(p-1,q-1)$ and $\LCM(p-1,q-1)$ are even, $\pi$ is a permutation with a unique fixed point; all remaining elements of $\Z_n$ are contained in cycles of even length; and there are an even number of cycles of even length. Moreover, by Lemma \ref{lem:goodlinear}(2), all elements of $[\pi]$ have this cycle structure.  Thus, $h(n)=(n-1)/2$ by Lemma \ref{lem:orthsaregood}(2).
\end{proof}

We get the following immediate corollary.

\begin{cor}
 \label{cor:pq1and3mod4}
 Let $n = pq$, where $p \equiv 1 \pmod 4$ and $q \equiv 3 \pmod 4$ are primes.  Then, $h(n) = (n-1)/2$.  In particular, there are infinitely many integers $n \equiv 3 \pmod 4$ such that $h(n) = (n-1)/2$.
\end{cor}

There is a more general version of Theorem \ref{thm:pqfulldiam} that can be applied to some values of $n$ that are divisible by more than two odd primes.

\begin{lem}
 \label{lem:ordersandorbits}
Let $n = p_1^{c_1} \cdots p_k^{c_k}$, where $p_1, \dots, p_k$ are distinct odd primes and each $c_i \ge 1$. Let $d$ be a divisor of $n$, and write $d = p_1^{e_1} \cdots p_k^{e_k}$, where $0 \le e_i \le c_i$ for each $i$. Let $I_d \colonequals \{i : e_i < c_i\}$, and let $X_d \colonequals \{x \in \Z_n : \gcd(x,n) = d\}$. 
\begin{enumerate}[(1)]
\item $|X_d| = \prod_{i \in I_d} p_i^{c_i - e_i - 1}(p_i - 1)$.

\item Let $a \in \Z_n^\times$, for each $1 \le i \le k$ let $m_i$ be the multiplicative order of $a$ modulo $p_i^{c_i-e_i}$, and let $m_d \colonequals \LCM(\{m_i : i \in I_d\})$. Define $\pi: \Z_n \to \Z_n$ by $x^\pi = xa$. Then, $\langle \pi \rangle$ has $|X_d|/m_d$ 
orbits of length $m_d$ 
on $X_d$. In particular, if each $m_i$ is even and $|I_d| > 1$, then $\langle \pi \rangle$ has an even number of orbits of even length on $X_d$.
\end{enumerate}
\end{lem}

\begin{proof}
(1) By the Chinese Remainder Theorem, every element $x \in \Z_n$ is determined uniquely by the $k$ congruences $x \equiv b_i \pmod {p_i^{c_i}}$, $1 \le i \le k$.  Let $x \in X_d$.  If $p_i^{c_i} \mid x$, then $b_i = 0$ and is unique; otherwise,  $\gcd(x, p_i^{c_i}) = p_i^{e_i} < p_i^{c_i}$, and there are $p_i^{c_i - e_i - 1}(p_i - 1)$ choices for $b_i$, since $x/p_i^{e_i}$ must be a unit modulo $p_i^{c_i-e_i}$.  Hence,
 \[ |X_d| = \prod_{i \in I_d} p_i^{c_i - e_i - 1}(p_i - 1).\]

(2) Let $x \in X_d$ and let $t \ge 1$ be the length of the orbit of $\langle \pi \rangle$ containing $x$. Then, $x^{\pi^t} = xa^t = x$, so $x(a^t - 1) = 0$ in $\Z_n$. This implies $a^t \equiv 1 \pmod {n/d}$. So, $m_i$ divides $t$ for each $i \in I_d$, which means $m_d$ divides $t$. Conversely, $a^{m_d} \equiv 1 \pmod {n/d}$, so $x(a^{m_d}-1) = 0$ in $\Z_n$, from which it follows that $x$ is in an orbit of length dividing $m_d$. Thus, every element of $X_d$ is in an orbit of length $m_d$, and there are $|X_d|/m_d$ 
orbits of $\langle \pi \rangle$ on $X_d$.

Finally, note that $\prod_{i \in I_d} m_i$ divides $|X_d|$ because the group of units modulo $p_i^{c_i-e_i}$ has order $p_i^{c_i-e_i-1}(p_i-1)$.  Since each $m_i$ is even, we can write $m_i = 2^{a_i}b_i$ with $a_i \ge 1$, and hence the exponent of the highest power of $2$ dividing $\prod_{i \in I_d} m_i$ is
\[ \sum_{i \in I_d} a_i \ge \max_{i \in I_d} a_i +(|I_d| - 1),\]
which is greater than the highest power of $2$ dividing $m_d$; hence $\left(\prod_{i \in I_d} m_i\right)/m_d$ is even.  Thus,
\begin{equation*}
\dfrac{|X_d|}{m_d} = \dfrac{|X_d|}{\prod_{i \in I_d} m_i} \cdot \dfrac{\prod_{i \in I_d} m_i}{m_d}
\end{equation*}
is even, as desired.
\end{proof}

\begin{thm}\label{thm: exponent sum}
Let $n = p_1^{c_1} \cdots p_k^{c_k}$, where $p_1, \dots, p_k$ are distinct odd primes and each $c_i \ge 1$. Assume that either of the following two conditions holds:
\begin{enumerate}[(i)]
\item $c_1 + \cdots + c_k$ is even.
\item $c_1 + \cdots + c_k$ is odd and there exists $1 \le \ell \le k$ such that $p_\ell \equiv 1 \pmod 4$ and $c_\ell$ is odd.
\end{enumerate}
Then, $h(n) = (n-1)/2$.
\end{thm}

\begin{proof}
We will specify a unit $a \in \Z_n^\times$ and define a permutation $\pi: \Z_n \to \Z_n$ by $x^\pi = xa$. The choice of $a$ depends on whether condition (i) or condition (ii) is satisfied.
\begin{enumerate}[(i)]
\item If condition (i) holds, then choose $a \in \Z_n$ such that for each $1 \le i \le k$, the multiplicative order of $a$ modulo $p_i^{c_i}$ is $p_i^{c_i-1}(p_i-1)$. This is possible because of the Chinese Remainder Theorem, and the fact that $\Z_{p_i^{c_i}}^\times$ is a cyclic group of order $p_i^{c_i-1}(p_i-1)$.
\item If condition (ii) holds, then $\Z_{p_\ell^{c_\ell}}^\times$ is cyclic of order $p_\ell^{c_\ell-1}(p_\ell-1)$ and $p_\ell-1$ is divisible by 4. So, we can choose $a \in \Z_n$ such that the multiplicative order of $a$ modulo $p_i^{c_i}$ is $p_\ell^{c_\ell-1}(p_\ell-1)/2$ when $i = \ell$, and $p_i^{c_i-1}(p_i-1)$ when $i \ne \ell$.
\end{enumerate}
Note that under either condition, if $1 \le i \le k$ and $0 \le e_i < c_i$, then the multiplicative order of $a$ modulo $p_i^{c_i-e_i}$ is even.

Let us first verify that $\pi^{-1}$ is an orthomorphism. Certainly, $\gcd(a,n)=1$ because $a$ is a unit of $\Z_n$. If $\gcd(a-1, n) > 1$, then $p_i \mid a - 1$ for some $i$, so $a \equiv 1 \pmod {p_i}$.  However, by construction, the multiplicative order of $a$ modulo $p_i$ is greater than $1$ for each $i$.  Thus $\gcd(a-1, n) = 1$. By Lemma \ref{lem:goodlinear}, $\pi^{-1}$ is an orthomorphism, and, moreover, all elements in $[\pi]$ have the same cycle structure as $\pi$.

The remainder of the proof consists of a description of the cycle structure of $\pi$. For each positive divisor $d$ of $n$, define $I_d$ and $X_d$ as in Lemma \ref{lem:ordersandorbits}. Then, $\Z_n = \bigcup_{d \mid n} X_d$, and this union is disjoint. So, the cycle structure of $\pi$ can be determined by examining the orbits of $\langle \pi \rangle$ on each set $X_d$.

Clearly, $X_n = \{0\}$ and 0 is a fixed point of $\pi$. Since $a$ is a unit of $\Z_n$, this fixed point is unique. So, $\langle \pi \rangle$ contains a single orbit of length 1. This occurs under both condition (i) and condition (ii).

Next, suppose $d$ is such that $|I_d| > 1$. Write $d=p_1^{e_1} \cdots p_k^{e_k}$ for some $0 \le e_i \le c_i$. Define the integers $m_i$ and $m_d$ as in Lemma \ref{lem:ordersandorbits}. As noted previously, each $m_i$ is even. So, by Lemma \ref{lem:ordersandorbits}, $\langle \pi \rangle$ has an even number of orbits of even length on $X_d$. As in the previous paragraph, this is true under both condition (i) and condition (ii).

It remains to consider the case where $|I_d|=1$. Let
\[D := \{d \in \N: d|n \text{ and } |I_d|=1\}\] and $X := \bigcup_{d \in D} X_d$. We will show that $\langle \pi \rangle$ has an even number of orbits of even length on $X$. Our arguments differ depending on whether condition (i) or condition (ii) is satisfied.

Assume first that condition (i) holds. Given $d \in D$, there exist $1 \le j \le k$ and $0 \le e_j < c_j$ such that $d=n/p_j^{c_j-e_j}$. Then,
\begin{equation*}
m_d=m_j=p_j^{c_j-e_j-1}(p_j-1) = |X_d|.
\end{equation*}
So, $\langle \pi \rangle$ has a single orbit of length $m_d$ on $X_d$. It remains in this case to determine $|D|$.  Now, once $j$ has been fixed, there are $c_j$ choices for $e_j$. Hence, $|D|=c_1 + \cdots + c_k$, which is even because condition (i) holds. Thus, across all of $X$, $\langle \pi \rangle$ has $|D|$ orbits, each of even length.

Finally, assume that condition (ii) holds. Without loss of generality, assume that $\ell=1$. If $d=n/p_{1}^{c_1-e_1}$ for some $0 \le e_1 < c_1$, then
\begin{equation*}
m_d=p_1^{c_1-e_1-1}(p_1-1)/2 = |X_d|/2.
\end{equation*}
So, $\langle \pi \rangle$ has 2 orbits of length $m_d$ on $X_d$. There are $c_1$ choices for $e_1$, so we obtain $2c_1$ orbits of length $m_d$ as $d$ runs through all divisors of $n$ of the form $n/p_{1}^{c_1-e_1}$. If $j \ne 1$ and $d=n/p_{j}^{c_j-e_j}$ for some $0 \le e_j < c_j$, then as under condition (i), $\langle \pi \rangle$ has a single orbit of length $m_d$ on $X_d$. So, under condition (ii), the total number of orbits of $\langle \pi \rangle$ on $X$ is $2c_1 + c_2 + \cdots + c_k$. By assumption, both $c_1$ and $c_1 + c_2 + \cdots + c_k$ are odd, so $2c_1 + c_2 + \cdots + c_k$ is even. We conclude that, once again, $\langle \pi \rangle$ has an even number of orbits of even length on $X$.

The theorem now follows by Lemma \ref{lem:orthsaregood}.
\end{proof}

From Theorems \ref{thm:pqfulldiam} and \ref{thm: exponent sum}, we see that there are infinitely many values of $n \equiv 3 \pmod{4}$ such that $h(n)=(n-1)/2$. However, there exist integers to which these theorems do not apply, such as primes congruent to $3 \pmod{4}$, or certain composite integers such as 27 or 75. For some of these values of $n$, we have been able to determine $h(n)$ via ad hoc or computational methods. We discuss these cases in the next two sections.

\section{Results for some specific $n \equiv 3 \pmod{4}$}\label{sec: specific n=3 mod 4}

In this section, we examine the value of $h(n)$ for some particular integers $n \equiv 3 \pmod{4}$ for which Theorem \ref{thm: exponent sum} does not apply. From Theorem \ref{thm: n=3 mod 4 bounds}, we know that $h(n)$ is equal to either $(n-3)/2$ or $(n-1)/2$. To conclude that $h(n)=(n-1)/2$, it is enough to find one $\pi \in A_n$ such that $h([\pi])=(n-1)/2$.

\begin{defn}\label{def: witness}
Let $n \ge 3$ be odd. We call $\pi \in A_n$ a \textit{witness} for $n$ if $\pi$ satisfies condition (i)--(iii) of Lemma \ref{lem:orthsaregood}(2).
\end{defn}

With this terminology, $h(n)=(n-1)/2$ if and only if a witness for $n$ exists. Furthermore, if $\pi$ is a witness for $n$, then $\pi$ is an orthomorphism on $\Z_n$, although the converse does not hold. For $n=7$, we can prove that no orthomorphism can be a witness, and hence conclude that $h(7)=2$.

\begin{lem}\label{lem:orthos in A_7}
There is no orthomorphism on $\Z_7$ having cycle structure 1-2-4 (i.e.\ having the form $(a_0)(a_1 , a_2)(a_3 , a_4 , a_5 , a_6))$.
\end{lem}
\begin{proof}
Suppose that $\pi = (a_0)(a_1 , a_2)(a_3 , a_4 , a_5 , a_6)$ is an orthomorphism on $\Z_7$ and let $\delta = \pi-\id$. Then, $\delta$ is a permutation on $\Z_7$ and $a_0^\delta = 0$. Let $d = a_1^\delta=a_2-a_1$, and note that $a_2^\delta = -d$. So, there exist $u, v \in \Z_7^\times$ such that $\{a_i^\delta : 3 \le i \le 6\} = \{ \pm u, \pm v\}$. Let $d_1 = \min\{ \pm u, \pm v\}$ (here, we view $\Z_7^\times$ as $\{1, 2, \ldots, 6\}$, so that a minimum value for $\{ \pm u, \pm v\}$ is well-defined). Without loss of generality, we may assume that $a_3^\delta = d_1$. Once this is done, let $d_2 = a_4^\delta$, which cannot equal $\pm d_1$. This forces $a_5^\delta = -d_1$ and $a_6^\delta=-d_2$. With this notation, $\pi$ has the form
\begin{equation*}
\pi = (a_0)(a_1, \, a_2)(a, \, a+d_1, \, a+d_1+d_2, a+d_2)
\end{equation*}
for some $a \in \Z_7$. Moreover, we may list the elements of $\Z_7$ as $a, a+1, \ldots, a+6$.

We will consider six cases depending on $d$, $d_1$, and $d_2$. These are summarized in Table \ref{Z7 Table}. In each case, after a choice has been made for $d$, the value of $d_1$ is forced because we have selected $d_1 = \min(\Z_7^\times \setminus \{\pm d\})$. Then, $d_2 \in \Z_7^\times \setminus\{ \pm d, \pm d_1\}$, so there are two choices for $d_2$. After specifying $d_1$ and $d_2$, we let $\sigma = (a, \, a+d_1, \, a+d_1+d_2, \, a+d_2)$ be the 4-cycle in $\pi$. Once $\sigma$ is known, $a_1$ and $a_2$ must come from the three elements of $\Z_7$ that are fixed by $\sigma$.

\begin{table}[ht]
\centering
\begin{tabular}{c|c|c|c|c}
$d$ & $d_1$ & $d_2$ & $\sigma$ & possible $a_1$, $a_2$\\
\hline
$\pm 1 \pmod{7}$ & 2 & $\begin{matrix} 3 \\ 4 \end{matrix}$ & $\begin{matrix} (a, \, a+2, \, a+5, \, a+3) \\ (a, \, a+2, \, a+6, \, a+4) \end{matrix}$ & $\begin{matrix} a+1,\, a+4,\, a+6\\ a+1,\, a+3,\, a+5\end{matrix}$ \\
\hline
$\pm 2 \pmod{7}$ & 1 & $\begin{matrix} 3 \\ 4 \end{matrix}$ & $\begin{matrix} (a, \, a+1, \, a+4, \, a+3) \\ (a, \, a+1, \, a+5, \, a+4) \end{matrix}$ & $\begin{matrix} a+2,\, a+5,\, a+6\\ a+2,\, a+3,\, a+6\end{matrix}$ \\
\hline
$\pm 3 \pmod{7}$ & 1 & $\begin{matrix} 2 \\ 5 \end{matrix}$ & $\begin{matrix} (a, \, a+1, \, a+3, \, a+2) \\ (a, \, a+1, \, a+6, \, a+5) \end{matrix}$ & $\begin{matrix} a+4,\, a+5,\, a+6\\ a+2,\, a+3,\, a+4\end{matrix}$ \\
\end{tabular}
\caption{Restrictions on an orthomorphism on $\Z_7$ with cycle structure 1-2-4.}
\label{Z7 Table}
\end{table}

In each case, it is impossible to choose $a_1$ and $a_2$ from among the possible choices and have $a_2=a_1+d$. We conclude that $\Z_7$ has no orthomorphism with cycle structure 1-2-4.
\end{proof}

\begin{thm}\label{thm:h(7)=(n-3)/2}
$h(7)=2$.
\end{thm}
\begin{proof}
By Theorem \ref{thm: n=3 mod 4 bounds}, either $h(7)=2$ or $h(7)=3$. If $h(7)=3$, then by Lemma \ref{lem:orthsaregood} there exists $\pi \in A_7$ such that each $\sigma \in [\pi]$ has exactly one fixed point, the cycle decomposition for $\sigma$ contains a single cycle of odd length, and the remaining elements are contained in an even number of disjoint cycles of even length. Since $n=7$, the only cycle structure satisfying these requirements is 1-2-4. Thus, $\pi$ must have this cycle structure. However, since each element of $[\pi]$ has a unique fixed point, $\pi^{-1}$ is an orthomorphism by Lemma \ref{lem:orthsaregood}. This contradicts Lemma \ref{lem:orthos in A_7}, because $\pi^{-1}$ also has cycle structure 1-2-4. Consequently, no such $\pi$ exists, and therefore $h(7)=2$.
\end{proof}

We are able to prove a similar result for $n = 11$ computationally. Our approach here is to first identify all orthomorphisms on $\Z_{11}$, and then show that none of these permutations is a witness for $n=11$.

\begin{thm}
 \label{thm:11}
 $h(11) = 4$.
\end{thm}

\begin{proof}
 By Theorem \ref{thm: n=3 mod 4 bounds}, either $h(11)=4$ or $h(11)=5$.  By computation in GAP \cite{GAP4}, there are exactly 22308 elements of $A_{11}$ that are orthomorphisms, and exactly 13310 of these orthomorphisms can be expressed as a product of a single cycle of length one along with an even number of disjoint cycles of even length.  However, for each of these elements $\pi$, the coset $[\pi]$ contains at least one permutation having more than one cycle of odd length. 
 By Lemma \ref{lem:orthsaregood}, this implies that $h(11) \neq 5$, i.e., $h(11) = 4$.
\end{proof}

For larger values of $n$, we have not found any examples for which $h(n)=(n-3)/2$, but nor have we found a general method to produce a witness for $n$. Overall, finding orthomorphisms with a given cycle structure seems to be a difficult problem; however, recent progress was made in \cite{BorsWang2022}, where the authors constructed complete mappings and orthomorphisms with a given cycle structure. Following the work done in \cite{BorsWang2022}, we have been able to construct witnesses for some larger values of $n$. The permutations produced by these techniques are piecewise functions that are linear on cosets of a subgroup of the unit group of $\Z_n$. We illustrate this type of function in Example \ref{ex: n=19} below.

\begin{ex}\label{ex: n=19}
Let $K$ be the subgroup of $\Z_{19}^\times$ generated by 8. Then, $K = \{8,7,18,11,12,1\}$ in $\Z_{19}^\times$. The cosets of $K$ are $K$ itself,
    $$
    2K = 3K = \{16,14,17,3,5,2\},
    $$
    and
     $$
    4K = 9K = \{13,9,15,6,10,4\}.
    $$ 
Define the following function $\pi$ on $\Z_{19}$:
    $$
    x^\pi := \begin{cases}
        0, & \text{ if }x=0 \\
        -x, & \text{ if }x\in K \\
        3x, & \text{ if }x\in 2K \\
        6x, & \text{ if }x\in 4K
    \end{cases}.
    $$
It is easily seen that $\pi$ is a bijection, and when written in cycle notation, $\pi$ is
\begin{equation*}
\pi = (0)(1, 18) (2,6,17,13)(3,9,16,10)(4,5,15,14)(7,12)(8,11).
\end{equation*}
One may check that $\pi$ is an orthomorphism on $\Z_{19}$, and is a witness for $n=19$. Thus, $h(19)=9$.
\end{ex}

In the spirit of Example \ref{ex: n=19} (and with the aid of computations in GAP), we have produced witnesses for various primes $n \equiv 3 \pmod{4}$, along with $n=27$. Witnesses for further values of $n$ were found by using the large language model Claude Fable 5. We discuss all results obtained with the help of AI models in Section \ref{sec: Claude}.

\begin{thm}
 \label{thm:largesporadicn}
 We have $h(n)=(n-1)/2$ for the following values of $n$:
 \begin{equation*}
 n = 19, 23, 27, 71, 103, 127, 271, 571, 631, 883, 991, 1039, 1171, 1279, 1579, 1723.
 \end{equation*}
\end{thm}

\begin{proof}
A witness for $n=19$ is shown in Example \ref{ex: n=19}. Permutations that serve as witnesses for the other values of $n$ are given in \url{https://github.com/EricSwartz/CircularSorting/blob/main/Witnesses.txt}. 
\end{proof}

\section{Results obtained with the aid of Claude}\label{sec: Claude}

All results in this section were obtained either in whole or in part by the large language models Claude Fable 5 and Claude Opus 5. The AI produced one notable general result (Theorem \ref{thm:compositewitnesstype}) and found witnesses for a number of values of $n$ beyond those mentioned in Section \ref{sec: specific n=3 mod 4}.

\begin{defn}\label{def: witness-type}
For all $\sigma \in S_n$, let $\sgn(\sigma)=1$ if $\sigma$ is even, and $\sgn(\sigma)=-1$ if $\sigma$ is odd. We say that a permutation $\pi$ in $S_n$ is \textit{witness-type} if all elements of the coset $\langle c \rangle \pi$ are permutations that have a unique fixed point and all other elements of $\{0, 1, \ldots, n-1\}$ are contained in cycles of even length.  Note that a witness-type permutation $\pi$ such that $\sgn(\pi) = 1$ is simply a witness in the sense of Definition \ref{def: witness}.  If a permutation $\pi$ is witness-type and $\sgn(\pi) = -1$, then we say that $\pi$ is a \textit{$(-1)$-witness}. 
\end{defn}

The following result is inspired by results that take complete mappings or orthomorphisms in subgroups and/or quotients to construct complete mappings or orthomorphisms in the larger group; see, e.g., \cite[Theorem 4]{Paige1951}, \cite[Corollary 2]{HallPaige1955}, \cite[Theorem 8.1]{Evans}, \cite[Lemma 2.1, Proposition 2.4]{AkhtarGagola}, \cite[Theorem 1.3]{BorsWang2022-3}, and especially the wreath product constructions in \cite[Section 4 and Lemma 28]{bastide2025circularsortingstrongcomplete}.

\begin{thm}
\label{thm:compositewitnesstype}
Let $n_1, n_2 \ge 3$ be odd and let $n=n_1 n_2$. Let $\sigma \in S_{n_1}$ and $\gamma \in S_{n_2}$ be permutations of witness-type. Then, there exists a permutation $\pi \in S_n$ of witness-type such that $\sgn(\pi) = \sgn(\sigma)\sgn(\gamma)$.
\end{thm}
\begin{proof}
We will let $S_n$ act on $\Z_{n_1} \times \Z_{n_2}$. Order $\Z_{n_1} \times \Z_{n_2}$ lexicographically, so that $(x_1, x_2) < (y_1, y_2)$ if and only if either $x_1 < y_1$ or $x_1=y_1$ and $x_2 < y_2$. Let $c \in S_n$ be the cycle corresponding to this order, i.e.\
\begin{equation*}
c = \big( (0,0), \, (0,1), \, \ldots, \, (0,n_2-1), \, (1,0), \ldots, (n_1-1, n_2-1) \big).
\end{equation*}
Then, for any $(x_1, x_2) \in \Z_{n_1} \times \Z_{n_2}$,
\begin{equation}\label{eq: action of c}
(x_1, x_2)^c = \begin{cases} (x_1, \, x_2+1), & \text{if } x_2 < n_2-1\\ (x_1+1, \, x_2+1), & \text{if } x_2 = n_2-1 \end{cases},
\end{equation}
where the first entry of each ordered pair is reduced modulo $n_1$, and the second entry is reduced modulo $n_2$. By iterating \eqref{eq: action of c}, we obtain a formula for the action of $c^k$ on $\Z_{n_1} \times \Z_{n_2}$ for any $k \ge 0$. First, given $k$, let $0 \le k_1 \le n_1-1$ and $0 \le k_2 \le n_2-1$ be such that $k=k_2 + k_1n_2$. Then, for all $(x_1, x_2) \in \Z_{n_1} \times \Z_{n_2}$,
\begin{equation}\label{eq: action of c^k}
(x_1, x_2)^{c^k} = \begin{cases} (x_1+k_1, \, x_2+k_2), & \text{if } x_2 + k_2 < n_2\\ (x_1+k_1+1, \, x_2+k_2), & \text{otherwise } \end{cases}.
\end{equation}

Now, define a permutation $\pi$ on $\Z_{n_1} \times \Z_{n_2}$ by $(x_1, x_2)^\pi = (x_1^\sigma, x_2^\gamma)$. Then, $\pi \in S_n$ and we will show that $\pi$ is of witness-type.

Let $k \ge 0$ and define $k_1$ and $k_2$ as above. For $i=1,2$, let $c_i = (0, \, 1, \, \ldots, \, n_i-1) \in S_{n_i}$. By \eqref{eq: action of c^k}, for each $(x_1, x_2) \in \Z_{n_1} \times \Z_{n_2}$ we have
\begin{align*}
(x_1, x_2)^{c^k \pi}  &= 
\begin{cases} 
\big( (x_1+k_1)^\sigma, \, (x_2+k_2)^\gamma \big), & \text{if } x_2 + k_2 < n_2\\
\big( (x_1+k_1+1)^\sigma, \, (x_2+k_2)^\gamma \big), & \text{otherwise} 
\end{cases}\\
&= 
\begin{cases} 
\big( x_1^{c_1^{k_1}\sigma}, \; x_2^{c_2^{k_2}\gamma} \big), & \text{if } x_2 + k_2 < n_2\\
\big( x_1^{c_1^{k_1+1}\sigma}, \; x_2^{c_2^{k_2}\gamma} \big), & \text{otherwise} 
\end{cases}.
\end{align*}
Since both $\sigma$ and $\gamma$ are of witness-type, each permutation $c_1^{k_1}\sigma$, $c_1^{k_1+1}\sigma$, and $c_2^{k_2}\gamma$ has a unique fixed point. Let $y_2 \in \Z_{n_2}$ be the fixed point of $c_2^{k_2}\gamma$, and note that $y_2$ depends only on $k$. Let $y_1 \in \Z_{n_1}$ be the fixed point of either $c_1^{k_1}\sigma$ or $c_1^{k_1+1}\sigma$, depending on whether or not $y_2 + k_2 < n_2$. Then, $(y_1, y_2)$ is fixed by $c^k\pi$. Moreover, if $(z_1, z_2) \in \Z_{n_1} \times \Z_{n_2}$ is fixed by $c^k\pi$, then we must have $z_2=y_2$. This implies that $z_1=y_1$. Thus, $c^k \pi$ has a unique fixed point in $\Z_{n_1} \times \Z_{n_2}$.

Next, let $(a_1, a_2) \in \Z_{n_1} \times \Z_{n_2}$. We will prove that $(a_1, a_2)$ is contained in an orbit even length under the action of $\langle c^k \pi \rangle$ unless $(a_1, a_2)$ is the unique fixed point $(y_1, y_2)$ of $c^k \pi$.  Let
\begin{equation*}
j_1 := \begin{cases} k_1, & \text{if } a_2+k_2 < n_2\\ k_1+1, & \text{otherwise } \end{cases}.
\end{equation*}
Then, for any $m \ge 0$,
\begin{equation}\label{eq: orbit length}
\big(a_1, a_2\big)^{(c^k \pi)^m} = \big(a_1^{(c_1^{j_1}\sigma)^m}, \; a_2^{(c_2^{k_2}\gamma)^m}\big).
\end{equation}
Assume now that
\[ \big(a_1, a_2\big)^{(c^k \pi)^m} = (a_1, a_2), \]
which by \eqref{eq: orbit length} implies that
\[ a_2^{(c_2^{k_2}\gamma)^m} = a_2.\]  We proceed by cases.  If $a_2 \neq y_2$, then $a_2^{c_2^{k_2}\gamma} \neq a_2$, and, since $\gamma$ is witness-type by assumption, this implies that $m$ is even and hence $(a_1, a_2)$ is contained in an orbit of even length under the action of $\langle c^k \pi \rangle$.  On the other hand, if $a_2 = y_2$ and $a_1 \neq y_1$, then
\[ \big(a_1, y_2\big) = \big(a_1, y_2\big)^{(c^k \pi)^m} = \big(a_1^{(c_1^{j_1}\sigma)^m}, \; y_2^{(c_2^{k_2}\gamma)^m}\big) = \big(a_1^{(c_1^{j_1}\sigma)^m}, \; y_2\big),\]
which implies that
\[ a_1^{(c_1^{j_1}\sigma)^m} = a_1.\]
Since $a_1 \neq y_1$ and $\sigma$ is witness-type by assumption, this implies $m$ is even and hence $(a_1, a_2)$ is contained in an orbit of even length under the action of $\langle c^k \pi \rangle$.  Hence, if $(a_1, a_2) \neq (y_1, y_2)$, $(a_1, a_2)$ lies in a cycle of $c^k \pi$ of even length. Since $k \ge 0$ was arbitrary, we conclude that $\pi \in S_n$ is of witness-type.

It remains to verify that $\sgn(\pi) = \sgn(\sigma)\sgn(\gamma)$. Note that $\pi = (\sigma, \id_2)(\id_1, \gamma)$, where $\id_i$ denotes the identity of $S_{n_i}$. Suppose that $\sigma$ can be written as a product of $t$ transpositions in $S_{n_1}$. Then, $(\sigma, \id_2)$ can be written as a product of $tn_2$ transpositions in $S_n$. Since $n_2$ is odd, we have $\sgn((\sigma, \id_2))=\sgn(\sigma)$. Similarly, $\sgn((\id_1,\gamma))=\sgn(\gamma)$. So, $\sgn(\pi) = \sgn(\sigma)\sgn(\gamma)$, which completes the proof.
\end{proof}

\begin{cor}\label{cor: Claude corollaries}\mbox{}
\begin{enumerate}[(1)]
\item Let $n_1, n_2 \ge 3$ be odd. If $h(n_1)=(n_1-1)/2$ and $h(n_2) = (n_2-1)/2$, then $h(n_1 n_2) = (n_1n_2-1)/2$.

\item Let $n \ge 3$ be odd and such that there exists both a witness and a $(-1)$-witness in $S_n$. Then, for any odd $k \ge 1$, $h(nk) = (nk-1)/2$.

\item Let $n$ be an odd integer with prime factors $p$ and $q$ such that $p \equiv 1 \pmod 4$ and $q \equiv 3 \pmod 4$.  Then, $h(n) = (n - 1)/2$.
\end{enumerate}
\end{cor}
\begin{proof}
(1) Assume $h(n_1)=(n_1-1)/2$ and $h(n_2) = (n_2-1)/2$. By Lemma \ref{lem:orthsaregood}, both $A_{n_1}$ and $A_{n_2}$ contain permutations of witness-type. By Theorem \ref{thm:compositewitnesstype}, $A_{n_1 n_2}$ contains a witness, so $h(n_1 n_2) = (n_1n_2-1)/2$.

(2) There is nothing to prove if $k=1$, so assume that $k \ge 3$. Since $k$ is odd, the map $\sigma: x \mapsto -x$ in $S_k$ is of witness-type; this follows mutatis mutandis from the proof of Proposition \ref{prop: n=1 mod 4}.  By hypothesis, there exists $\gamma \in S_n$ of witness-type such that $\sgn(\gamma) = \sgn(\sigma)$. By Theorem \ref{thm:compositewitnesstype}, $A_{nk}$ contains a witness.

(3) On the one hand, $h(pq) = (pq - 1)/2$ by Corollary \ref{cor:pq1and3mod4}, so $S_{pq}$ contains a witness for $pq$.  On the other hand, the map $x \mapsto -x$ is a $(-1)$-witness for $pq$. The result follows from part (2).
\end{proof}

\begin{rem}\label{rem: remaining cases}
By Corollary \ref{cor: Claude corollaries}(3) and Theorem \ref{thm: exponent sum}, we see that if there exists a counterexample $n$ to Conjecture \ref{conj: n=3 mod 4}, then $n$ has the form $n = p_1^{c_1} \cdots p_k^{c_k}$, where each $p_i$ is a prime congruent to $3 \pmod{4}$ and $c_1+\cdots+c_k$ is odd.
\end{rem}

Note that the permutations $\sigma$ and $\gamma$ in Theorem \ref{thm:compositewitnesstype} may be odd, even though the resulting permutation $(\sigma, \gamma)$ in $S_{n_1 n_2}$ is even. In the next example, we demonstrate odd permutations on $\Z_{3}$ and $\Z_{21}$ that can be used to produce a witness for $n=63$.
\begin{ex}\label{ex: n=63}
Let $n_1=3$, $n_2=21$, and $\sigma = (1,2) \in S_3$. In $S_{21}$, the map $\gamma_1: x \mapsto -x$ is witness-type with $\sgn(\gamma_1) = 1$ (see Proposition \ref{prop: n=1 mod 4}). Moreover, Claude Opus 5 produced the permutation 
\begin{equation*}
\gamma_2 = (0)(1,13,8,6,11,15,9,12,2,3,17,14)(4,10,19,18)(5,7,20,16),
\end{equation*}
which can be verified to be a $(-1)$-witness in $S_{21}$. By Corollary \ref{cor: Claude corollaries}(2), $h(63) = (63 - 1)/2=31$. A specific witness for $n=63$ can be calculated from $\sigma \in S_3$ and $\gamma_2 \in S_{21}$ as in the proof of Theorem \ref{thm:compositewitnesstype}.  The resulting permutation, as an element of $S_{63}$, is
\begin{align*}
\pi:=(0)&(1, 13, 8, 6, 11, 15, 9, 12, 2, 3, 17, 14)(4, 10, 19, 18)(5, 7, 20, 16)(21, 42)\\
&(22, 55, 29, 48, 32, 57, 30, 54, 23, 45, 38, 56)(24, 59, 35, 43, 34, 50, 27, 53, 36, 51, 33, 44)\\
&(25, 52, 40, 60)(26, 49, 41, 58)(28, 62, 37, 47)(31, 61, 39, 46)
\end{align*}
and $\pi$ is a witness for $n = 63$.
\end{ex}

Claude was also able to find witnesses for additional primes $p \equiv 3 \pmod{4}$. Much like the orthomorphisms studied in \cite{BorsWang2022}, these permutations are piecewise linear and defined by using subgroups and cosets in the unit group of $\Z_p$. We show the resulting permutation for $p=31$.

\begin{ex}\label{ex: n=31}
Let $K = \langle 26 \rangle = \{26,25,30,5,6,1\} \le \Z_{31}^\times$. The cosets of $K$ are $K$ itself,
    \begin{align*}
         3K &= \{16,13,28,15,18,3\},\\
         9K &= \{17,8,22,14,23,9\},\\
         19K &= \{29,10,12,2,21,19\}, \\
         27K &= \{20,24,4,11,7,27\}.\\
    \end{align*}
Define a permutation $\pi$ on $\Z_{31}$ by
    $$
    x^\pi := \begin{cases}
        0, & \text{ if }x=0 \\
        -x, & \text{ if }x\in K\\
        15x, & \text{ if }x\in 3K \\
        22x, & \text{ if }x\in 9K\\
        -2x, & \text{ if }x\in 19K \\
        20x, & \text{ if }x\in 27K 
    \end{cases}.
    $$
In cycle notation,
\begin{align*}
\pi = (0)(1, &30) (2,27,13,9,12,7,16,23,10,11,3,14,29,4,18,22,19,24,15,8,21,20,28,17)\\
& (5,26)(6,25).
\end{align*}
The verification that $\pi$ is a witness for $31$ is a straightforward calculation in GAP and is left to the reader.
\end{ex}

\begin{thm}\label{thm: Claude's witnesses}
We have $h(n)=(n-1)/2$ for the following values of $n$:
\begin{align*}
n = 31, \,&43, 67, 79, 131, 139, 151, 163, 191, 199, 211, 223, 283, 307, 331, \\
&367, 379, 439, 463, 487, 499, 523, 547, 607, 619, 643, 691, 727,\\
&739, 751, 787, 811, 823, 859, 907, 919, 967.
\end{align*}
\end{thm}
\begin{proof}
See the witnesses given in \url{https://github.com/EricSwartz/CircularSorting/blob/main/ClaudeWitnesses.txt}.
\end{proof}

\section{Concluding remarks}

We end with some questions and remarks.

\begin{itemize}
 \item The exact circular sorting number in $A_n$ is still unknown for infinitely many values of $n$. The smallest prime for which $h(n)$ is unknown is $n=47$, the smallest composite integer for which $h(n)$ is unknown is $n=99$, and the values of $n$ less than $100$ for which $h(n)$ is unknown are $n = 47, 59, 83,$ and $99$.  It would be interesting to determine whether the sorting numbers for these values are also consistent with Conjecture \ref{conj: n=3 mod 4}.
 
 \item In the case of circular sorting with non-adjacent transpositions, exact results are known only in the case of primes and for a very restricted number of composite integers.  In contrast, the $n \equiv 3 \pmod 4$ case has been resolved for a large number of composite integers but is wide open for the case of prime integers in general (and more or less reduces to the case of primes congruent to $3 \pmod 4$; see Remark \ref{rem: remaining cases}).

 \item Even if Conjecture \ref{conj: n=3 mod 4} is false, it would be interesting to determine whether there are infinitely many primes $p \equiv 3 \pmod 4$ such that $h(p) = (p-1)/2$.  One possible path to resolving this would be to prove that specific permutations constructed as in \cite{BorsWang2022} (or Example \ref{ex: n=31}) work infinitely often.  Perhaps it is possible to accomplish this by taking a fixed value of $d$ (say $d = 3$) and a fixed \textit{special permutation} (say $\psi = (0)(1, \, 2)$; see \cite[Section 2]{BorsWang2022}) and proving that the permutations constructed this way give rise to a witness infinitely often.  Since by \cite[Theorem 1.1.1]{BorsWang2022} the permutations constructed this way will be orthomorphisms for $n$ sufficiently large, such an approach is not without hope.

 \item In order to fully resolve Conjecture \ref{conj: n=3 mod 4} using techniques similar to those of the present paper, it is necessary to construct orthomorphisms where the cycle structure is restricted for all elements of the coset $[\pi]$.  This appears to be quite difficult in general, especially since our computational work has exhibited witnesses $\pi$ such that (in contrast to Lemma \ref{lem:goodlinear}) elements of $[\pi]$ have differing cycle structures.

\end{itemize}

\section*{Declaration of generative AI use}

We used Claude Fable 5 to check our arguments and find typos.  It suggested a correction for an error in the last paragraph of the proof of Lemma \ref{lem:ordersandorbits}, and this was incorporated into the paper.  Additionally, the results given in Section \ref{sec: Claude} were obtained with the assistance of Claude Fable 5 and Claude Opus 5.  All other arguments in the paper are due to the authors, and the authors assume responsibility for all content in the paper.  \\

\bibliographystyle{plainurl}
\bibliography{references}

\end{document}